\documentclass[12pt]{article} 

\usepackage{amsmath}
\usepackage{amsthm}
\usepackage{amsfonts}
\usepackage{mathrsfs}
\usepackage{stmaryrd}
\usepackage{setspace}
\usepackage{fullpage}
\usepackage{amssymb}
\usepackage{breqn}
\usepackage{enumitem}
\usepackage{bbold} 
\usepackage{authblk}
\usepackage{comment}
\usepackage{hyperref}
\usepackage{pgf,tikz}
\usepackage{graphicx}
\usepackage{subcaption}

\newtheorem{thm}{Theorem}[section]

\newtheorem{proposition}[thm]{Proposition}

\newtheorem{corollary}[thm]{Corollary}

\newtheorem{clm}[thm]{Claim}

\newcommand\ex{\ensuremath{\mathrm{ex}}}

\newcommand\cN{{\mathcal N}}

\newcommand{\ignore}[1]{}

\title{Generalized Turán problems for $K_{2,t}$}
\author{Dániel Gerbner \\ Alfr\'ed R\'enyi Institute of Mathematics}
\date{}

\begin{document}

\maketitle

\begin{abstract}
    We study the generalized Tur\'an function $\ex(n,H,F)$, when $H$ or $F$ is $K_{2,t}$. We determine the order of magnitude of $\ex(n,H,K_{2,t})$ when $H$ is a tree, and determine its asymptotics for a large class of trees. We also determine the asymptotics of $\ex(n,K_{2,t},H)$ in most cases.
\end{abstract}

\section{Introduction}

One of the most fundamental problems of extremal graph theory deals with determining the largest number of edges in $n$-vertex graphs that do not have a given subgraph $F$. This quantity is called the Tur\'an number of $F$ and is denoted by $\ex(n,F)$. Here we mention one particular result due to  F\"uredi \cite{fur}: for $t\ge 2$ we have $\ex(n,K_{2,t})=(1+o(1))\sqrt{t-1}n^{3/2}/2$.

In this paper we deal with a generalization of the Tur\'an number. Given two graphs $H$ and $G$, we denote by $\cN(H,G)$ the number of (unlabeled) copies of $H$ in $G$. We let $\ex(n,H,F):=\max \{\cN(H,G): \,G \text{ is an $n$-vertex $F$-free graph}\}$, i.e. the largest number of copies of $H$ in $n$-vertex $F$-free graphs. In particular, $\ex(n,K_2,F)=\ex(n,F)$.
We will consider the cases when either $H$ or $F$ is $K_{2,t}$ for some $t\ge 2$. 

Let us describe the above mentioned $K_{2,t}$-free construction of F\"uredi \cite{fur} in more detail, as we will use it later. Let $q_t(n)$ be the largest prime power such that $t-1$ divides $q_t(n)-1$ and $(q_t(n)^2-1)/(t-1)\le n$. When $t$ and $n$ are clear from the context, we will sometimes omit them and denote $q_t(n)$ by $q$. It is well-known that for sufficiently large $n$ there exists such a $q_t(n)$ with $\sqrt{nt}-n^{1/3}<q_t(n)$ \cite{HI}. Therefore, we can use graphs on $(q_t(n)^2-1)/(t-1)$ vertices instead of $n$-vertex graphs, and this does not change the asymptotics of the lower bounds obtained.

Let us consider a $q_t(n)$-element field and let $h$ be an element of order $t-1$, thus the elements $\{h,h^2,h^3,\dots,h^{t-1}\}$ form a multiplicative subgroup. Two pairs $(a,b)$ and $(a',b')$ of elements are considered equivalent if $a=h^pa'$ and $b=h^pb'$ for some $p$. Then we have $(q_t(n)^2-1)/(t-1)$ equivalence classes, they form the vertex set of our graph. The equivalence class of $(a,b)$ and the equivalence class of $(c,d)$ are joined by an edge if $ac+bd=h^p$ for some $p$. Observe that by this definition some vertices could be connected to themselves, creating a loop; we will call such vertices \textit{special} and ignore the loops to obtain a simple graph that we call the F\"uredi graph and denote by $F(n,t)$.

F\"uredi \cite{fur} showed that any special vertex of $F(n,t)$ is connected to exactly $q-1$ vertices and other vertices are connected to exactly $q$ vertices. Two vertices $u$ and $v$ of $F(n,t)$ have exactly $t-1$ common neighbors unless $u$ and $v$ are adjacent and at least one of them is a special vertex, in which case they have at most $t-2$ common neighbors. 

\smallskip

Now we are ready to introduce the main definition of this paper.
A graph $H$ is called \textit{$t$-F\"uredi-good} if \[\ex(n,H,K_{2,t})=(1+o(1))\cN(H,F(n,t)).\] 

This definition is motivated by the similar notion of $k$-Tur\'an-good graphs \cite{gp2}. Those are the graphs $H$ with $\ex(n,H,K_k)=\cN(H,T_{k-1}(n)$, where $T_{k-1}(n)$ is the $(k-1)$-partite Tur\'an graph, which is the unique $n$-vertex $K_k$-free with $\ex(n,K_k)$ edges. Similarly, here we examine when we have the same extremal graph as in the ordinary Tur\'an problem, but only in the asymptotical sense.

Now, we can quickly summarize most of the results concerning $\ex(n,H,K_{2,t})$ for $t\ge 2$: Alon and Shikhelman \cite{as} showed that $K_3$ is $t$-F\"uredi-good, Gerbner and Palmer \cite{GP2019} showed that $P_k$ and $C_k$ are $t$-F\"uredi-good, Gerbner and Patk\'os \cite{gepat} showed that $K_{2,s}$ is $t$-F\"uredi-good.

There are some other results. Zhang and Ge \cite{zg} showed that for $t\ge 2m-3\ge 3$ we have $\ex(n,K_m,K_{2,t})=\Theta(n^{3/2})$.
Gerbner and Patk\'os \cite{gepat} determined $\ex(n,K_{1,p},K_{2,t})$ exactly if $p>t+1$ and $n$ is large enough.
The case of forbidden $K_{2,2}$ was studied in \cite{gerbner2,ggymv}. In particular, Gerbner \cite{gerbner2} showed that $P_3$ is 2-F\"uredi-good, but also presented an exact result $\ex(n,P_3,K_{2,2})=\binom{n}{2}$ for $n$ even and $\ex(n,P_3,K_{2,2})=\binom{n}{2}-1$ (with the construction being a matching added to the star $S_n$). Moreover, for other stars we have $\ex(n,S_r,K_{2,2})=\binom{n-1}{r-1}$.


\medskip

In our main result, we characterize the $t$-F\"uredi-good trees. Furthermore, we determine the order of magnitude of $\ex(n,T,K_{2,t})$ for other trees as well. Before stating our results, we need to introduce further definitions. 

Let $T$ be a tree. We will partition its vertex set into two parts $A$ and $B$ applying a greedy process. Originally $A$ consists of vertices of degree 2, and $B$ contains the rest of the vertices. Then, if there is a non-leaf vertex $v$ in $B$ that is connected to at most two vertices in $B$, then we add $v$ to $A$. We repeat this as long as we can find a vertex satisfying the above property. This gives an ordering $v_1,\dots,v_m$ of vertices in $A$ with degree more than 2 (the order of adding them to $A$). This ordering is not unique, but each ordering gives the same $A$ and $B$. When we refer to this ordering later, we mean that we fix one of the possible orderings.

We say that $T$ is \textit{nice} if every vertex of $B$ is a leaf, and $T$ has at least 3 vertices.

\begin{thm}\label{nicetrees}
Nice trees are $t$-F\"uredi-good for every $t\ge 2$.
\end{thm}

Let us assume now that $T$ is not nice and introduce some notation. Let $L$ be the set of leaves of $T$.
Let $Q$ be the subgraph induced on $B$ (thus $Q$ contains $L$). Let $Q'$ be the subgraph we obtain by deleting $L$ from $Q$. 
$Q'$ is a forest, let $Q_i$ denote its connected components. 
Removing $Q'$ cuts $T$ into several subtrees $T_1,\dots,T_s$. Each $T_j$ is connected to one or more $Q_i$ with an edge.
Observe that if $Q'$ has no vertices, then $T$ is nice.

\begin{thm}\label{nemnice}
Let $t\ge 2$ and $T$ be a tree on at least three vertices that is not nice. Then $\ex(n,T,K_{2,t})=\Theta\left(n^{|L|+\frac{|A|+s}{2}}\right)$.
\end{thm}

We will show later that the F\"uredi graph contains  $\Theta\left(n^{\frac{|V(T)|+1}{2}}\right)=\Theta\left(n^{\frac{|A|+|L|+1}{2}}\right)$ copies of $T$, thus for non-nice trees $\cN(T,F(n,t))$ does not even have the same order of magnitude as $\ex(n,T,K_{2,t})$.


\medskip

Let us consider now the problem of counting copies of $K_{2,t}$.
Several results of Gerbner and Patk\'os contain this as a subcase; in particular they determined the exact value of $\ex(n,K_{2,t},K_{p,q})$ if $2<p,q\le t$ and $n$ is large enough.

Gy\H ori, Pach and Simonovits \cite{gyps} showed that if $H$ is complete multipartite, then $\ex(n,H,K_k)=\cN(H,G)$ for some complete $(k-1)$-partite graph $G$ on $n$ vertices. $G$ does not have to be balanced, but they showed that if $H$ is $K_{2,2}$ or $K_{2,3}$, then $G$ is balanced, i.e. the Tur\'an graph. It is a straightforward optimization to find the optimal graph for given values of $t$ and $k$, but we are unable to deal with this problem in this generality. Ma and Qiu \cite{mq} showed that for $k=3$, in the case of $t=4$ the Tur\'an graph is not optimal, but is asymptotically optimal. In other words, $\ex(K_{2,4},K_3)\neq \cN(K_{2,t},K_{\lfloor n/2\rfloor,\lceil n/2\rceil})$ but $\ex(K_{2,4},K_3)=(1+o(1)) \cN(K_{2,t},K_{\lfloor n/2\rfloor,\lceil n/2\rceil})$. For larger $t$, $K_{\lfloor n/2\rfloor,\lceil n/2\rceil}$ does not even give the correct asymptotics.

Gerbner and Palmer \cite{GP2019} proved that if $F$ has chromatic number $k$, then $\ex(n,H,F)\le \ex(n,H,K_k)+o(n^k)$ for any $H$. In our case it shows that the complete $(k-1)$-partite $n$-vertex graph $G$ with the most copies of $K_{2,t}$ gives the correct asymptotics for $\ex(n,K_{2,t},F)$. 

Gerbner, Z. Nagy and Vizer \cite{gnv} determined the asymptotics of $\ex(n,K_{2,t},C_{2k})$, generalizing a result of Gerbner, Gy\H ori, Methuku and Vizer \cite{ggymv}. There are too many results on $\ex(n,C_4,F)$ to list them here.

Alon and Shikhelman \cite{as} determined the graphs $F$ with $\ex(n,K_3,F)=O(n)$. Their results also implied a similar characterization where $K_3$ is replaced with any forest. Gerbner and Palmer \cite{GP2019} obtained such a characterization for every cycle.

In this paper we obtain such a characterization for $K_{2,t}$. Moreover, we characterize the possible orders of magnitude below $n^t$. Let $K_{2,r}^{p,q}$ denote the graph we obtain from $K_{2,r}$ if we connect $p$ new vertices to $u$ and $q$ new vertices to $v$, where $u$ and $v$ are the vertices of the partite set of size 2 of $K_{2,r}$. We note that $p,q,r$ can be 0. For integers $k,t\ge 2$, we let $M(k,t)$ denote the complete $k$-partite graph with the most copies of $K_{2,t}$. Clearly, $M(k,t)$ contains $\Theta(n^{t+2})$ copies of $K_{2,t}$. For a graph $H$, we denote by $kH$ the graph consisting of $k$ vertex-disjoint copies of $H$.

\begin{thm}\label{cou}
For any integer $t\ge 2$ and any graph $F$ we have

\begin{displaymath}
\ex(n,K_{2,t},F)=
\left\{ \begin{array}{l l}
0 & \textrm{if\/ $K_{2,t}$ contains $F$},\\
(1+o(1))\cN(K_{2,t},\lfloor\frac{n}{p+q+r+1}\rfloor K_{p+q+r+1}) & \textrm{if\/ $F=K_{2,r}^{p,q}$ with $r\le t$ and $p+q+r>t$},\\
(1+o(1))\cN(K_{2,t},F(q(n),r)) & \textrm{if\/ $F=K_{2,r}^{p,q}$ with $r> t$},\\
\Omega(n^t) & \textrm{if\/ $F$ is any other bipartite graph},\\
(1+o(1))\cN(K_{2,t},M(k-1,t)) & \textrm{if\/ $F$ has chromatic number at least three}.\\
\end{array}
\right.
\end{displaymath}
\end{thm}

Let us discuss the missing case.
For a bipartite graph $F$, we denote by $\beta(F)$ the smallest $p$ such that $F$ is a subgraph of $K_{p,q}$ for some $q$. In the above theorem, the $\Omega(n^t)$ bound belongs to the case $\beta(F)\ge 3$. If $\beta(F)\le t$, we have $\ex(n,K_{2,t},F)=\Theta(n^t)$, where the upper bound follows from a result of Gerbner and Patk\'os \cite{gepat} on $\ex(n,K_{2,t},K_{p,q})$.
If $\beta(F)< t$, we have $\ex(n,K_{2,t},F)=(1+o(1))\cN(K_{2,t},K_{\beta(F)-1,n-\beta(F)+1})$, where the upper bound follows from another result of Gerbner and Patk\'os \cite{gepat}. This also implies that for any given $F$, for $t\ge |V(F)|-2$ we have the asymptotics of $\ex(n,K_{2,t},F)$ (assuming in the case $\chi(F)\ge 3$ that we can solve the optimization problem mentioned earlier).

If
$\beta(F)>t$, then by a result of Alon and Shikhelman \cite{as} we have the upper bound $O(n^{t+2-2t/\beta})$. This is also sharp in some cases, see \cite{gepat} for a collection of results on $\ex(n,K_{s,t},K_{p,q})$.

\smallskip

In Section 2, we deal with the case where $K_{2,t}$ is forbidden, we prove Theorems \ref{nicetrees} and \ref{nemnice}, and a corollary. In Section 3 we deal with the case of counting copies of $K_{2,t}$ and prove Theorem \ref{cou} through a series of propositions. Some of those prove slightly more than stated in Theorem \ref{cou}.

\section{Forbidding $K_{2,t}$}

In this section instead of copies of $T$, we will often talk about embeddings of $T$, that create ordered copies of $T$. As it only depends on $T$ how many ordered copies of itself $T$ contains, it will not change anything when we compare the number of copies of $T$ in $G$ and in $F(n,t)$. More precisely, an embedding $f:V(T)\rightarrow V(G)$ is an injective function such that if $uv$ is an edge of $T$, then $f(u)f(v)$ is an edge of $G$. Let $\cN'(T,G)$ denote the number of embeddings of $T$ into $G$. Then $\cN'(T,G)=c(T)\cN(T,G)$ for some $c(T)$ depending only on $T$ (the number of automorphisms). We will also talk about embeddings with some fixed vertices. That means that $f(v_1), \dots, f(v_\ell)$ are already given. More precisely, if $v_1,\dots,v_\ell\in V(T)$ and $x_1,\dots, x_\ell\in V(G)$, then let $\cN'(T,v_1,\dots,v_\ell,G,x_1,\dots,x_\ell)$ denote the number of embeddings $f$ of $T$ into $G$ such that $f(v_i)=x_i$ for every $i\le \ell$.

Let us define embeddings $f$ of $T$ into $F(n,t)$, 
greedily. First we let $f(v)=x$ for some $v\in V(T)$ and $x\in V(F(n,t))$, and then in each step, we pick a vertex of $T$ adjacent to exactly one of the vertices already embedded. It is well-known that we can build any tree this way. When we pick a vertex $y$ that is adjacent to an already embedded vertex $z$, we need to pick a neighbor $f(z)$ in $G$. We pick a neighbor that we have not picked as an image.
Clearly, there are $(1+o(1))\sqrt{(t-1)n}$ ways to pick such an image. The number of embeddings we found is $(1+o(1))(t-1)^{(|V(T)|-1)/2}n^{(|V(T)|+1)/2}$. 


We state the following strengthening of Theorem \ref{nicetrees}. This will help us with the induction. Note that $K_2$ is not a nice tree, but it is $t$-F\"uredi-good. However, the strengthening below does not hold for $K_2$.

\begin{thm}\label{main}
Let $T$ be a nice tree. Then $\ex(n,T,K_{2,t})=(1+o(1))\cN(T,F(n,t))$. Moreover, let $\ell\ge 1$ be an integer and $v_1,\dots,v_\ell$ be leaves of $T$ such that any pair of them has distance more than 2. Then for any $K_{2,t}$-free $n$-vertex graph $G$ and $x_1,\dots, x_\ell\in V(G)$, we have 
$\cN'(T,v_1,\dots,v_\ell,G,x_1,\dots,x_\ell)\le (1+o(1))(t-1)^{(|V(T)|-1)/2}n^{(|V(T)|-2\ell+1)/2}$.
\end{thm}

\begin{proof}

We use induction on $|V(T)|$. 
We deal first with the case $T$ has a vertex $w$ such that each leaf is at distance two from $w$ (i.e. $T$ is a spider with legs of length two). This proves the base case $|V(T)|=3$. First we embed the leaves that are not fixed and $w$, arbitrarily, this can be done at most $n^{\frac{|V(T)|+1}{2}-\ell}$ ways. Then each of the other $(|V(T)|-1)/2$ vertices is the common neighbor of $w$ and a leaf, thus can be embedded to at most $t-1$ vertices.


Let us assume that $|V(T)|\ge 4$.
We will define an embedding $g$ of $T$ into $G$ recursively such that $g(v_i)=x_i$ for every $i\le \ell$. It is easy to see that each such embedding of $T$ into $G$ can be found with the recursive embeddings defined below, thus an upper bound on the number of such embeddings is an upper bound on $\cN'(T,v_1,\dots,v_\ell,G,x_1,\dots,x_\ell)$.

Let us consider a subtree $T'$ with a vertex $v\in A$ such that a leaf neighbor $v'$ of $v$ is already embedded. If $v$ has another leaf neighbor, we denote that neighbor by $v''$. By deleting $v$ from $T'$ we obtain several subtrees $T_i'$. We label a subtree with the most vertices to be $T_1'$. Each subtree $T_i'$ has a vertex $w_i$ adjacent to $v$. If $v''$ does not exist, we can assume that $T_1'$ has more than 2 vertices (otherwise $T$ is a spider with legs of length two and we are done).

Now we are ready to define the recursive embedding $g$ of $T'$.
We first embed $v''$
arbitrarily. If there is no $v''$, we will embed $T_1'$ first. There are two possibilities. If there is a fixed vertex $u'$ in $T_1'$, we use that as the first embedded vertex. Then $u'$ is a leaf and has a neighbor $u$. Then $u$ is not a leaf, thus $u\in A$, hence we can apply the recursion with $u$ playing the role of $v$ and $u'$ playing the role of $v'$, and embed $T_1'$.

If $T_1'$ does not have a fixed vertex, then we pick a vertex $w_1'$ of $T_1'$ in $A$ with a leaf neighbor $w_1''$. Such vertices exist as $T_1'$ has leaves and also has at least one vertex ($w_1$) in $A$. We embed $w_1''$ first. Then $w_1'$ has an already embedded leaf neighbor, and at most one other neighbor in $B$, thus we can apply the recursion and embed $T_1'$.

After embedding $v''$ or $T_1'$, we pick $g(v)$ as a common neighbor of $g(v')$ and $g(v'')$, or if $v''$ does not exist, we pick $g(v)$ as a common neighbor of $g(v')$ and $g(w_1)$. Consider a tree $T_i'$ not yet embedded. We have $w_i\in A$, and $w_i$ has a neighbor $v$ already embedded. In the tree that we obtain by adding $v$ to $T_i'$, $v$ is a leaf, thus we can embed  $T_i'$ recursively.

Let us count the embeddings. 
By induction, we know the situation in each $T_i'$. Assume that $T_i'$ contains $\ell_i$ fixed vertices. For $i>1$, we embed $T_i'$ together with $v$, thus we have $|V(T_i')|+1$ vertices, and $\ell_i+1$ of them are fixed. Therefore, there are at most $(1+o(1))(t-1)^{|V(T_i')|/2}n^{(|V(T_i')|-2\ell_i)/2}$ embeddings of $T_i'$. 

If $v''$ exists, we have the same bound for $T_1'$, and a factor of at most $t-1$ for choosing $g(v)$ as a common neighbor of $g(v')$ and $g(v'')$. Furthermore, if $v''$ is not fixed, we have an additional factor of at most $n$ for choosing $v''$. 
The $\ell_i$'s add up to $\ell-1$ if $v''$ is not fixed an to $\ell-2$ if $v''$ is fixed. The $|V(T_i')|$'s add up to $|V(T)|-3$, thus we obtain the desired bound. 

If $v''$ does not exist, then the $|V(T_i')|$'s add up to $|V(T)|-2$. We can embed $T_1'$ at most $(1+o(1))(t-1)^{(|V(T_1')|-1)/2}n^{(|V(T_1')|-2\ell_1+1)/2}$ ways. Again, we have a factor of $t-1$ for choosing $g(v)$, which gives the desired bound. 
\end{proof}

The above theorem easily implies that we can attach nice trees to $t$-F\"uredi-good graphs to obtain other $t$-F\"uredi-good graphs.

\begin{corollary}
Let $H$ be a $t$-F\"uredi-good graph, $T$ be a nice tree and $H'$ be obtained from $H$ and $T$ by identifying an arbitrary vertex of $H$ and a leaf of $T$. Then $H'$ is $t$-F\"uredi-good.
\end{corollary}

\begin{proof}
First we embed $H$, and then embed $T$ with one of its leaves fixed. By definition and by Theorem \ref{main}, in both cases $F(n,t)$ asymptotically maximizes the number of embeddings.
\end{proof}

Let us turn to other trees. Recall that Theorem \ref{nemnice} states that $\ex(n,T,K_{2,t})=\Theta(n^{|L|+\frac{|A|+s}{2}})$.



\begin{proof}[Proof of Theorem \ref{nemnice}]
Let us start with the upper bound and let $G$ be an $n$-vertex $K_{2,t}$-free graph. We will take a copy of $T$ the following way. First we pick the leaves adjacent to $Q'$, we have at most linearly many choices for each. If a vertex of $T$ is connected to at least 2 vertices picked earlier, then there is a constant number of ways to pick them. We pick those vertices. Assume that no such vertex is left. We claim that we have picked every vertex of $Q'$. Recall that $Q'$ is a forest. After we have picked some of its vertices, the remaining vertices form a forest $Q''$. If $Q''$ is not empty, then there is a leaf $w$ of $Q''$. Observe that $w$ is a non-leaf vertex of $T$ and $w\in B$. This shows that $w$ is connected to at least 3 vertices of $B$, and at most one of those is in $Q''$. Hence $w$ has at least two neighbors already picked, thus we also pick $w$. This shows that we picked every vertex of $B$, and we had $O(n^\ell)$ choices.

Let $T_i$ be a subtree with neighbors $v_1,\dots,v_\ell$, then let $T_i'$ be $T_i$ together with $v_1,\dots,v_\ell$. The number of ways to pick $T_i'$ with $v_1,\dots,v_\ell$ fixed is $\Theta(n^{(|V(T_i')|-2\ell+1)/2})$ 
by Theorem \ref{main}. The $|V(T_i')|$'s add up to $|A|$, and we have an extra $n^{1/2}$ factor for every $i$, which completes the proof of the upper bound.

To prove the lower bound, we define a $K_{2,t}$-free graph $G_0$. We let $n'=(n-|V(Q')|)/|V(T)$. We take a copy of $Q'$, and one copy of each $T_j$. Consider a vertex $v$ of $Q'$. If $v$ has a leaf neighbor, we add $n'$ new vertices and connect them to $v$. 
Let us consider the subtrees $T_j$ now and let $v_1,\dots,v_\ell$ be the neighbors of vertices of $T_j$ in $Q'$. We take a copy of $F(n',t)$, and we will pick $\ell$ independent vertices in it. We identify those $\ell$ vertices with $v_1,\dots,v_\ell$. We pick the $\ell$ independent vertices of $F'(n,t)$ such a way that it maximizes the number of embeddings of $T_j$ with the vertices $v_i$ connected to the appropriate vertices of the embedded copy of $T_j$. We repeat this for every $j$ (with different copies of $F(n',t)$).

The resulting graph $G_0$ has at most $n$ vertices as we replace at most $|V(T)|$ vertices with $n'$ new vertices.
It is also easy to see that $G_0$ is $K_{2,t}$-free. Indeed, the only cycles in $G_0$ are inside the vertex-disjoint copies of $F(n',t)$, which are $K_{2,t}$-free.

Let us count the number of embeddings of $T$ into $G_0$. For each leaf adjacent to a vertex of $Q'$, we have linearly many choices. 

We will show that the number of ways to pick $T_j'$ in $F(n',t)$ is $\Omega(n^{(|V(T)|-2\ell+1)/2})$, completing the proof. Indeed, without fixing the $\ell$ leaves there are $\Theta(n^{(|V(T)|+1)/2})$ ways to pick $T_j'$ by Theorem \ref{main}, thus there must be a way to fix the leaves such that we can pick the remaining part of $T'_j$ (i.e. $T_j)$ in $\Omega(n^{(|V(T)|-2\ell+1)/2})$ ways.
\end{proof}

\section{Counting $K_{2,t}$}
In the section we prove Theorem \ref{cou}. 
First we deal with the linear case, where we also obtain an exact result for infinitely many values of $n$. Recall that $K_{2,r}^{p,q}$ is obtained by connecting $p$ and $q$ leaves to the two vertices in the partite set of size two of $K_{2,r}$.

\begin{proposition}\label{nembipklikk}
Let $r\le t$ and $p+q+r> t$. Then $\ex(n,K_{2,t},K_{2,r}^{p,q})=\lfloor\frac{n}{p+q+r+1}\rfloor\cN(K_{2,t},K_{p+q+r+1})+O(1)$. Moreover, if $p+q+r+1$ divides $n$, then $\ex(n,K_{2,t},K_{2,r}^{p,q})=\frac{n}{p+q+r+1}\cN(K_{2,t},K_{p+q+r+1})$.
\end{proposition}

\begin{proof}
The lower bound is obtained by $\lfloor\frac{n}{p+q+r+1}\rfloor$ vertex disjoint copies of $K_{p+q+r+1}$. 

For the upper bound, consider an $n$-vertex $K_{2,r}^{p,q}$-free graph $G$. We will define an auxiliary multigraph $G'$ on the same vertex set. Two vertices $u$ and $v$ are connected by $m$ edges if there are $m$ copies of $K_{2,t}$ in $G$ where $u$ and $v$ form the partite set of size 2. Clearly, the number of edges in $G'$ is equal to the number of copies of $K_{2,t}$ in $G$. Therefore, it is enough to bound the average degree in $G'$.

We claim that it is enough to
show that the average degree in $G'$ is at most $\binom{p+q+r-1}{t}(p+q+r)$.
Indeed, if $G$ consists of $\lfloor\frac{n}{p+q+r+1}\rfloor$ vertex disjoint copies of $K_{p+q+r+1}$, then the average degree in $G'$ is at most $\binom{p+q+r-1}{t}(p+q+r)$, with equality if $p+q+r+1$ divides $n$.

Observe that for a vertex $v$ with degree $d$ in $G$, its degree is at most $\binom{d}{t}(p+q+r-2)$ in $G'$, as we pick $t$ neighbors of $v$, and then those $t$ neighbors have at most $p+q+r-1$ common neighbors, we pick one of them that is different from $v$. 

\begin{clm}
If $d_G(v)\ge p+q+r$, then in $G'$ at most $(p+q+r)\binom{p+q+r-1}{t}$ edges go from $v$ to vertices $w$ with $d_G(w)\ge p+r$.
\end{clm}

\begin{proof}[Proof of Claim]
 If $v$ is connected to a vertex $w$ in $G'$, this means that $v$ and $w$ have $t$ common neighbors in $G$. Assume first that $vw\not\in E(G)$. If $d_G(w)\ge p+r$, then we pick $p$ common neighbors of $v$ and $w$, $r$ other neighbors of $w$, and $v$ still has $q$ neighbors left. This way we obtain a $K_{2,t}^{p,q}$, a contradiction. Thus $d_G(w)\le p+r-1$.
 
Assume now that $vw\in E(G)$. If $d_G(v)\ge p+q+r+1$, the above reasoning still works and we obtain that $d_G(w)\le p+r-1$. Finally, let us assume that $d_G(v)=p+q+r$ and $w$ is one of the $p+q+r$ neighbors of $v$ in $G$. Then there are at most $\binom{p+q+r-1}{t}$ copies of $K_{2,t}$ where $v$ and $w$ are in the partite set of size 2, i.e. there are at most $\binom{p+q+r-1}{t}$ edges between $v$ and $w$ in $G'$.
\end{proof}

Let us return to the proof of the proposition. Let $A$ denote the set of vertices with degree at least $p+q+r$ in $G$, and $B$ denote the set of vertices with degree at most $p+r-1$. Then in $G'$, the vertices of $B$ have degree at most $\binom{p+r-1}{t}(p+q+r-2)$. For every vertex $v$ of $A$, in $G'$ all but $\binom{p+q+r-1}{t}(p+q+r)$ edges incident to $v$ go to a vertex of $B$. This implies that in $G'$ all but $\binom{p+q+r-1}{t}(p+q+r)|A|$ edges incident to $A$ go to $B$, thus there are at most $\binom{p+q+r-1}{t}(p+q+r)|A|+\binom{p+r-1}{t}(p+q+r-2)|B|$ edges incident to $A\cup B$. This shows that the average degree in $G'$ is at most $\binom{p+q+r-1}{t}(p+q+r)$ on $A\cup B$. The other vertices have degree at most $p+q+r-1$ in $G$, thus degree at most $\binom{p+q+r-1}{t}(p+q+r-2)$ in $G'$, finishing the proof.
\end{proof}

We remark that it is easy to extend this proof to show that in a $K_{2,r}^{p,q}$-free graph with $\ex(n,K_{2,t},K_{2,r}^{p,q})$ copies of $K_{2,t}$, all but $O(1)$ of the vertices form vertex-disjoint copies of $K_{p+q+r+1}$. Indeed, all but $O(1)$ vertices must be in $A$ and have $p+q+r-1$ common neighbors (as other vertices have smaller degree in $G'$). In other words, we take almost $n/(p+q+r+1)$ copies of $K_{p+q+r+1}$, and then the graph with the most copies $K_{2,t}$ on the remaining $O(1)$ vertices. What is left is to determine $\ex(n,K_{2,t},K_{2,r}^{p,q})$ for small values of $n$.

A natural idea is to take $\lfloor n/(p+q+r+1)\rfloor$ copies of $K_{p+q+r+1}$ and a smaller clique on the remaining vertices, but this construction is not always optimal. Consider e.g. $\ex(14,K_{2,7},K_{2,2}^{3,3})$. The above construction contains 36 copies of $K_{2,7}$, while another $K_{2,2}^{3,3}$-free graph $K_{7,7}$ contains 42 copies.

\begin{proposition}\label{kvad}
Let $r>t>1$. Then $\ex(n,K_{2,t},K_{2,r}^{p,q})=(1+o(1))\cN(K_{2,t},F(q(n),r))=(1+o(1))\binom{n}{2}\binom{r-1}{t}$.
\end{proposition}
\begin{proof}
Let $G$ be an $n$-vertex $K_{2,r}^{p,q}$-free graph.
 Assume first that two vertices $u$ and $v$ have degree at least $r+p+q$. Then $u$ and $v$ have at most $r-1$ common neighbors, thus they are in the smaller partite set of at most $\binom{r-1}{t}$ copies of $K_{2,t}$. 
 
 Assume now that the degree of a vertex $u$ is less than $r+p+q$. Then for every $v$ such that $u$ and $v$ form the smaller partite set of a $K_{2,t}$, there are at least $t$ neighbors of $u$ connected to $v$. There are at most $\binom{p+q+r}{t}$ ways to pick $t$ neighbors of $u$. These $t$ neighbors have at most $p+q+r$ common neighbors. Therefore, there are at most $(p+q+r)\binom{p+q+r}{t}$ vertices that can be connected to $t$ neighbors of $u$, thus $u$ is in the smaller partite set of at most $(p+q+r)\binom{p+q+r}{t}$ copies of $K_{2,t}$.
 
 We count the copies of $K_{2,t}$ by picking the two vertices in the smaller partite set. Let us count first the copies where the two vertices in the smaller partite set
 have degree at least $r+p+q$. They are in the smaller partite set of at most $\binom{r-1}{t}$ copies of $K_{2,t}$, thus there are at most $\binom{n}{2}\binom{r-1}{t}$ such copies of $K_{2,t}$. Finally, we count the copies of $K_{2,t}$ where one of the vertices in the smaller partite set has degree less than $p+q+r$. There are at most $n$ such vertices, and each of them is contained in the smaller partite set of at most $(p+q+r)\binom{p+q+r}{t}$ copies of $K_{2,t}$. Therefore, there are at most $\binom{n}{2}\binom{r-1}{t}+n(p+q+r)\binom{p+q+r}{t}$ copies of $K_{2,t}$ in $G$, finishing the proof.
\end{proof}

Now we are ready to prove Theorem \ref{cou}.

\begin{proof}[Proof of Theorem \ref{cou}]

The first line of the inequality is obvious, the second and third lines are dealt with in Propositions \ref{nembipklikk} and \ref{kvad}. The fourth line is shown by the construction $K_{2,n-2}$.

For non-bipartite graphs $F$, the lower bound holds by definition and the upper bound follows from the results mentioned in the introduction: the result of Gy\H ori, Pach and Simonovits \cite{gyps} stating that a complete $(k-1)$-partite graph contains the most copies of any complete multipartite $H$ among $K_k$-free graphs, and a result of Gerbner and Palmer \cite{GP2019} stating that changing $K_k$ to any $k$-chromatic $F$ results in an additive term $o(n^{|V(H)|})$. 
\end{proof}


\vskip 0.3truecm

\textbf{Funding}: Research supported by the National Research, Development and Innovation Office - NKFIH under the grants KH 130371, SNN 129364, FK 132060, and KKP-133819,  and by the Ministry of Education and Science of the Russian Federation in the framework of MegaGrant no 075-15-2019-1926.

\end{document}